\documentclass[12pt,a4paper]{article}
\usepackage{amsmath,amsthm,amssymb,amsfonts,color,graphicx}

\newcommand{\C}{{\mathbb{C}}}          
\newcommand{\Octoni}{{\mathbb{O}}}     
\newcommand{\R}{{\mathbb{R}}}          

\newcommand{\rr}{\rightarrow}
\newcommand{\lrr}{\longrightarrow}

\newcommand{\calR}{{{\cal R}^\xi}}             %

\newcommand{\na}{{\nabla}}

\newcommand{\Aut}[1]{{\mathrm{Aut}}\,{#1}}
\newcommand{\End}[1]{{\mathrm{End}}\,{#1}}

\newcommand{\grad}{{\mathrm{grad}}\,}
\newcommand{\dx}{{\mathrm{d}}}

\newcommand{\inv}[1]{{#1}^{-1}}
\newcommand{\papa}[2]{\frac{\partial#1}{\partial#2}}

\newcommand{\cinf}[1]{{\mathrm{C}}^\infty_{#1}}

\newcommand{\expo}{{\mathrm{e}}}

\newtheorem{teo}{Theorem}[section]

\newtheorem{coro}{Corollary}[section]
\newtheorem{prop}{Proposition}[section]

\newenvironment{Rema}[1][Remark.]{\begin{trivlist}
\item[\hskip \labelsep {\bfseries #1}]}{\end{trivlist}}

\def\cyclic{\mathop{\kern0.9ex{{+}
\kern-2.2ex\raise-.28ex\hbox{\Large\hbox{$\circlearrowright$}}}}\limits}

\title{Weighted metrics on tangent sphere bundles}

\author{R. Albuquerque\footnote{{\texttt{rpa@uevora.pt}}\ ,\ \ \ Departamento de Matem\'atica da Universidade de \'Evora and Centro de Investiga\c c\~ao em Matem\'atica e Aplica\c c\~oes (CIMA), Rua Rom\~ao Ramalho, 59, 671-7000 \'Evora, Portugal.}}

\begin{document}


\maketitle


\markright{\sl\hfill  R. Albuquerque \hfill}

\begin{abstract}

Natural metric structures on the tangent bundle and tangent sphere bundles $S_rM$ of a Riemannian manifold $M$ with radius function $r$ enclose many important unsolved problems. Admitting metric connections on $M$ with torsion, we deduce the equations of induced metric connections on those bundles. Then the equations of reducibility of $TM$ to the almost Hermitian category. Our purpose is the study of the natural contact structure on $S_rM$ and the $G_2$-twistor space of any oriented Riemannian 4-manifold.

\end{abstract}

\vspace*{4mm}


\vspace*{2mm}


\vspace*{10mm}

The author acknowledges the support of Funda\c{c}\~{a}o Ci\^{e}ncia e Tecnologia, Portugal, through Centro de Investiga\c c\~ao em Matem\'atica e Aplica\c c\~oes da Universidade de \'Evora (CIMA-UE) and the sabbatical grant SFRH/BSAB/895/2009.


\vspace*{5mm}

\subsection{Introduction}

This article is the first part of a study of the geometry of tangent sphere bundles $S_rM=\{u\in TM:\ \|u\|=r\}$ of a Riemannian manifold $(M,g)$ with variable radius and weighted Sasaki metric.

It is today well established that any oriented Riemannian 4-manifold $M$ gives rise to a canonical $G_2$ structure on $S_1M$. This was discovered in \cite{Alb2,AlbSal1,AlbSal2} partly recurring to twistor methods; so we call it the $G_2$-twistor bundle of $M$. Indeed, the pull-back of the volume form coupled with each point $u\in S_1M$, say a 3-form $\alpha$, induces a quaternionic structure which is reproduced twice in horizontal and vertical parts of $T_uS_1M$. Then the Cayley-Dickson process gives the desired $G_2=\Aut{\Octoni}$-structure over $S_1M$. Some properties of the so called \textit{gwistor space} have been discovered, namely that it is cocalibrated if and only if the 4-manifold is Einstein. The first variation of that structure, which may yield interesting features, is by choosing both any metric connection (i.e. with torsion) or a different weigh on both the horizontal and vertical sides of the Sasaki metric. Another open problem in the theory resides in understanding a certain tensor $\calR\alpha$ which consists of a derivation of $\alpha$ by the curvature of $M$. The $G_2$-representation theory on the gwistor space adds further perspectives upon the well known $SO(4)$ theory of metric compatible tensors on $M$. Henceforth we were led to a study of isometries of tangent sphere bundles (\cite{Alb4}). A study of the curvature of $S_rM$ will appear as the second part of this work. 

Throughout, we assume that $M$ is an $m$-dimensional manifold. We start by viewing a rather personal construction of the tangent bundle $TM\stackrel{\pi}\rr M$, i.e. the $2m$-dimensional point-vector manifold which governs most of the differential geometry of $M$. 

Next we assume a Riemannian metric $g$ and a compatible metric connection $\na$ on $M$. The latter induces a splitting of $TTM=H\oplus V$ with both $H,V$ parallel and isometric to $\pi^*TM$, the pull-back bundle. We review the known classification of $g$-natural metrics on $TM$ by \cite{Abb1,AbbCalva,AbbSarih} and continue our study assuming metrics of the kind $f_1\pi^*g\oplus f_2\pi^*g+f_3\mu\otimes\mu$ on $H\oplus V$, where $f_1,f_2,f_3$ are certain $\R$-valued functions on $M$ and $\mu$ is a canonical 1-form. 

We have in view the study of the induced metric on the tangent sphere bundle $S_rM$ with variable radius function $r\in\cinf{M}$. There exist Einstein metrics in some examples, precisely with those metrics for which $f_3\neq0$.

We proceed with the weighted metric $g^{f_1,f_2}=f_1\pi^*g\oplus f_2\pi^*g$ with $f_1,f_2>0$. Recall the Sasaki metric is just $g^S=g^{1,1}$ with $H$ induced by the Levi-Civita connection. We construct an almost complex structure $I^G$ and the associated symplectic structure on $TM$, first announced in \cite{Alb4} without proofs. In studying the equations of integrability, the roles of the functions $f_1,f_2$ are clearly distinguished. We deduce the torsion $T^\na$ must be of a precise vectorial type. As a corollary we find that the functions only have to be both constant, the curvature flat and the torsion zero, if and only if we require the structure on $TM$ to be K\"ahler. 

The canonical symplectic structure of $T^*M$ arising from the Liouville form is here related, implying further understanding of the contact geometry of the (co-)tangent sphere bundle. Long before $G_2$-twistor space, Y. Tashiro showed $S_1M$ admits a canonical metric contact structure. We present here a complete generalization of this result.

Parts of this article were written during a sabbatical leave at Philipps-Universit\"at Marburg. The author wishes to thank their hospitality and expresses his gratitude to Ilka Agricola, from Philipps-Universit\"at, also for her careful reading of many parts of this article.

This article is dedicated to Ilda Figueiredo, member of the European Parliament and Teacher.


\subsection{Differential geometry of the tangent bundle}

Let $M$ be an $m$-dimensional smooth manifold. Suppose we are given two charts $(U_1,\phi_1)$ and $(U_2,\phi_2)$, two points $x\in U_1,\ y\in U_2$ and two vectors $v_1,v_2\in\R^m$. Then we may define an equivalence relation between these objects:
\begin{equation}
 (U_1,\phi_1,x,v_1)\sim (U_2,\phi_2,y,v_2)\ \Leftrightarrow\ \left\{\begin{array}{l}
x=y,\\ \dx(\phi_2\circ\inv{\phi_1})_{\phi_1(x)}(v_1)=v_2    \end{array}  \right. .
\end{equation}
Notice that it might happen $U_1\cap U_2\neq\emptyset$. This equivalence relation gives place to a new finite dimensional manifold
\begin{equation}
 TM=\frac{\bigsqcup U\times\R^m}{\sim}
\end{equation}
by gluing all the charts $U\times\R^m$ (the set of charts arising from a covering of $M$ in the same atlas being sufficient, within the same differentiable structure). Charts of lower differentiable class clearly induce the same $TM$, which is called the tangent vector bundle.

We define a tangent vector $X_x$ at the point $x$ to be the class $X_x=[(\phi,x,v)]$ given a chart $\phi$ on a neighbourhood of $x$ and a $v\in\R^m$. In particular we have $TM=\bigcup_{x\in M} T_xM$, where $T_xM$ is the tangent space at $x$, the set of tangent vectors at $x$, naturally endowed with the structure of Euclidean space. In particular, one usually denotes $\partial_i$ or $\papa{ }{\phi^i}(x)=[\phi,x,e_i]$ when $e_i$ is a vector from the canonical basis of $\R^m$. We also write shortly $u\in TM$ to refer to a vector, without mentioning the base point $x$ to which it corresponds. There is a bundle projection $\pi:TM\rr M$ which stands for this relation, $\pi(u)=x$. We may also see $\inv{\pi}(x)=T_xM$.

Given two manifolds $M,N$ and a smooth map $f$ between them, the classes defined above are known to correctly transform under a map, called differentiation and denoted $\dx f$. It is defined from the tangent bundle $TM$ into $f^*TN$ and essentially described by
\begin{equation}
 \dx f([\phi,x,v])=[\psi,f(x),\dx(\psi\circ f\circ\inv{\phi})_{\phi(x)}(v)]
\end{equation}
with the obvious notation.

As a manifold, $TM$ has its own tangent vector bundle $TTM\rr TM$. If we differentiate $\pi$, then $V=\ker\dx\pi$ is the vertical bundle tangent to $TM$. There are canonical induced charts around a point $u=[\phi,x,v]$; we may then write $T_uTM$ vectors as
$[\phi\times1_m,[\phi,x,v],(v_1,v_2)]$, which demonstrates the existence of a canonical embedding of $T_xM$ as $V_u\subset T_uTM$ (the set of all tangent vectors such that $v_1=0$ does not vary with the charts). Hence we have a canonical identification of $V=\pi^*TM$ and a short exact sequence
\begin{equation}\label{sequenciaexata}
 0\lrr V\lrr TTM\stackrel{\dx\pi}\lrr\pi^*TM\lrr0  .
\end{equation}

As it is also well known, a vector field $X$ is a section of the tangent bundle. The tangent bundle $TTM$ has a canonical vector field denoted $\xi$. It is defined by $\xi_u=u$, thus defined as a vertical vector field.

If we want to differentiate $X$ in various ways and directions and compare the results, then it is useful to have a \textit{linear connection} in order to respect the vector bundles in which the derivatives appear. We thence suppose we have a connection $\na$ on $M$. Then
\begin{equation}
 H=\{X\in TTM:\ \pi^*\na_X\xi=0\}
\end{equation}
is a complement for $V$. Indeed, picking a chart such that $\na_i\partial_j=\sum_l\Gamma_{ij}^l\partial_l$, defining the Christoffel symbols of the connection, and writing shortly $X_u=[(x,v),(v_1,v_2)],\ \xi=[(x,v), (0,v)]$, with $v=\sum v^je_j,\ v_1=\sum a_1^je_j,\ v_2=\sum a_2^je_j$, we find 
\[\pi^*\na_X\xi=\sum_i\,\dx v^i(X)\partial_i +v^i\na_{\dx\pi(X)}\partial_i= 
\sum_i(a_{2}^i+\sum_{j,k} v^ja_1^k\Gamma_{kj}^i)\partial_i . \]
Notice that, if $X\in V$, then $a_1^k=0,\ \forall k$, so that $\pi^*\na_X\xi=\sum a_2^i\partial_i=X$ (we have abbreviated $\partial_i$ for $\pi^*\partial_i$). Thus clearly the $m$-dimensional kernel $H$ is a complement for $V$. Moreover, $\pi^*\na_\cdot\xi$ is the vertical projection onto $V$. For any vector field $X$ over $TM$ we may always find the unique decomposition ($\na^*$ denotes the pull-back connection)
\begin{equation}
 X=X^h+X^v=X^h+\na^*_X\xi.
\end{equation}
As a corollary to these observations, for a para\-me\-trized curve $\gamma\subset M$ we have that
\begin{equation}
\gamma\mbox{ is a geodesic of }\na\ \Leftrightarrow\ \dot{\gamma}\mbox{\ is horizontal, i.e.}\ \ddot{\gamma}\in H.
\end{equation}
Indeed, $\pi\dot{\gamma}=\gamma$, so the chain rule gives $\dx\pi(\ddot{\gamma})=\dot{\gamma}$ and thence we have $\ddot{\gamma}=[(\gamma,\dot{\gamma}),(\dot{\gamma},\ddot{\gamma})]$ in any given chart. Finally the equation $\na^*_{\ddot{\gamma}}\xi=0$, taking from above and introducing the chart components, becomes $\ddot{\gamma}^i+\sum_{j,k} \dot{\gamma}^j\dot{\gamma}^k\Gamma_{kj}^i=0$, which is the equation of geodesics.

Now, $\dx\pi$ induces an isomorphism between $H$ and $\pi^*TM$, cf. (\ref{sequenciaexata}), and we have $V=\pi^*TM$, this being by definition the kernel of $\dx\pi$. Hence we may define a vector bundle endomorphism 
\begin{equation}\label{aplicacaotheta}
 \theta:TTM\lrr TTM
\end{equation}
sending $X^h$ to the respective $\theta X^h\in V$ and sending $V$ to 0. We also define an endomorphism, denoted $\theta^t$, which gives $\theta^tX^v\in H$ and which annihilates $H$. In particular $\theta^t\theta X^h=X^h$ and $\theta^2=0$. We remark that the role of the morphism $\theta$ is not considered by other authors studying the tangent bundle. Sometimes we call $\theta X^h$ the mirror image of $X^h$ in $V$. The map $\theta$ was first used in \cite{Alb2,AlbSal1,AlbSal2}.

Another main instrument to use in our study, adapted from the theory of twistor spaces, is given as follows. We endow $TTM$ with the direct sum connection $\na^*\oplus\na^*$, which we denote simply by $\na^*$ or even just $\na$. We have that $\na^*\theta=\na^*\theta^t=0$.
\begin{Rema}
Away from the zero section, i.e. on $TM\backslash M$, we have a line bundle $\R\xi\subset\pi^*TM$. Notice the canonical section can be mirrored by $\theta^t$ to give another canonical vector field $\theta^t\xi$ and therefore a line bundle too, sub-bundle of $H$. This canonical horizontal vector field $\theta^t\xi$ is called the \textit{spray} of the connection in \cite{Dom,Sakai} or called \textit{geodesic field} in the more recent \cite{Geig}. It has the further property that $\dx\pi_u(\theta^t\xi)=u,\ \forall u\in TM$.
\end{Rema}

\subsection{Natural metrics on $TM$} 

Suppose the previous manifold $M,\na$ is also furnished with a Riemannian metric $g$. We also use $\langle\ ,\ \rangle$ in place of the symmetric tensor $g$; this same remark on notation is valid for the pull-back metric on $\pi^*TM$. We recall from \cite{Sasa} the now called Sasaki metric in $TTM=H\oplus V$: it is given by $g^S=\pi^*g\oplus\pi^*g$. With it, $\theta_|:H\rr V$ is an isometric morphism and $\theta^t$ corresponds with the adjoint endomorphism of $\theta$. We stress that $\langle\ ,\ \rangle$ on tangent vectors to the tangent bundle, with $V\perp H$, always refers to the Sasaki metric.

With the canonical vector field $\xi$ we may produce other symmetric tensors over $TM$: first the linear forms
\begin{equation}\label{etaemu}
 \xi^\flat\qquad\mbox{and}\qquad\mu=\xi^\flat\circ\theta
\end{equation}
and then the three symmetric products of these.
\begin{Rema}
In fact one may see that the 1-form $\mu$ does not depend on the chosen metric connection (it is the pull-back of the Liouville form on the co-tangent bundle under the musical isomorphism, cf. section \ref{comandsympstru}).
\end{Rema}

The classification of all natural metrics on $TM$ induced from $g$ may be found in \cite{Abb1,AbbCalva,AbbSarih}. An analysis of the convexity properties has shown that the metrics correspond with six weight functions $f_1,\ldots,f_6$ which depend only on $\|u\|^2_g,\ \,u\in TM$. So we assume the $f_i:[0,+\infty)\rr\R$ below are composed with the squared norm. First let, $\forall X,Y\in TTM$,
\begin{equation}
 \hat{g}(X,Y)=g^S(\theta X,Y)+g^S(X,\theta Y)
=\langle\theta X,Y\rangle+\langle\theta Y,X\rangle.
\end{equation}
This is a metric of signature $(n,n)$. Also let 
\begin{equation}
g^{f_1,f_2}=f_1\pi^*g\,\oplus\,f_2\pi^*g
\end{equation}
so $g^S=g^{1,1}$.

The referred classification may be written quite easily in the present setting. Following \cite[Corollary 2.4]{AbbSarih}, the statement is that every natural metric on $TM$ is given by
\begin{equation}\label{naturalmetricsTM}
 G=g^{f_1,f_2}+f_3\hat{g}+f_4\xi^\flat\otimes\xi^\flat+f_5\xi^\flat\odot\mu+f_6\mu\otimes\mu
\end{equation}
with further conditions, inequalities, on those functions to assure $G$ is positive definite. 
The interested reader may see properties of $G$ in general in \cite{Abb1,AbbSarih,BenLouWood1,BenLouWood2,Dom,KowSek2,MusTri,Oproiu} and other references therein. One of the peculiar natural metrics is the Cheeger-Gromoll metric:
$G^{{\text{C-G}}}=g^{1,f_2}+f_2\xi^\flat\otimes\xi^\flat$ with $f_2=\frac{1}{1+\|u\|^2},\ \,u\in TM$, and this has been studied by quite a few authors, cf. \cite{BenLouWood2,KozNied,Munteanu}.

\subsection{Some connections on $TM$}

Let $M$ be a Riemannian manifold of dimension $m=n+1$ with $n\geq1$ and let us continue to denote the metric by $g=\langle\ ,\ \rangle$ and the linear connection by $\na$. From now on we assume the connection is metric, which implies $\na^*g^S=0$. Let $r\in\cinf{M}$ be a function on $M$. Then we may consider the tangent sphere bundle of radius $r$
\begin{equation}
 S_rM=\{u\in TM:\ \|u\|_g^2=r^2\}.
\end{equation}
It is a $2n+1$-dimensional submanifold of $TM$, which carries a canonical contact structure for certain metrics. This was found by Y. Tashiro in \cite{Tash} and will be dealt with later. We refer the reader to \cite{Blair} for a state of the art on this development.

We shall be interested in the case of $r$ constant and thus on the metrics $G$ defined in (\ref{naturalmetricsTM}) for which we may write $T_uS_rM=u^\perp$, the $G$ orthogonal subspace. Since now $\xi^\flat=0$ on the hypersurface and $r$ is constant, it is not hard to see that we are referring only to metrics of the form
\begin{equation}
 G=g^{f_1,f_2}+f_3\mu\otimes\mu\qquad\mbox{with}\quad f_1,f_2,f_3\in\cinf{M}\ \ \mbox{and}
\end{equation}
with the functions $f_1,f_2,f_3$ (obviously we let these functions be composed with $\pi$ on the right hand side when used on the manifold $TM$) such that $f_1,f_2>0$ and $f_1+f_3>0$. We thus assume 
\begin{equation}\label{f1ef2}
 f_1=\expo^{2\varphi_1},\qquad f_2=\expo^{2\varphi_2}
\end{equation}
for some functions $\varphi_1,\varphi_2$ on $M$.

If $\na$ is a metric connection for $g$, i.e. makes $g$ parallel, then it is well known that $\na^{f_1}=\na+C_1$, with
\begin{equation}\label{TransfLeviCivitaundercg}
C_1(X,Y)=X(\varphi_1)Y+Y(\varphi_1)X-\langle X,Y\rangle\grad\varphi_1,
\end{equation}
is a metric connection for $f_1g$ on $M$ with the same torsion as $\na$ (cf. 
\cite[Theorem 1.159]{Besse}). We denote
\begin{equation}
 X(\varphi)=\dx\varphi(X)=\langle\grad\varphi,X\rangle.
\end{equation}

On $TM$ we define the function $\partial\varphi(u)=\dx\varphi_{\pi(u)}(u),\ \forall u$. In other words,
\begin{equation}
 \partial\varphi=\langle\theta\pi^*\grad\varphi,\xi\rangle
\end{equation}
where $\theta$ is the map introduced in (\ref{aplicacaotheta}). In particular, $\na^{*,f_1}=\na^*+\pi^*C_1$ makes $f_1\pi^*g$ parallel on $H$ and 
$\na^{*,f_2}_XY=\na_X^*Y+\theta\pi^*C_2(X,\theta^tY)$ makes $f_2\pi^*g$ parallel on $V$. Now the analysis of $\na^{*,f_1}(f_3\mu\otimes\mu)$ gives a quite complicated expression. It simplifies if we assume $f_3=0$ or both $f_1,f_3$ are constant.
\begin{prop}
Consider the linear connection $\tilde{D}^*=\na^{*,f_1}\oplus\na^{*,f_2}$ over $TM$.\\
(i) If $f_3=0$, then $g^{f_1,f_2}$ is parallel for the connection $\tilde{D}^*$.\\
(ii) If $f_1,f_3$ are constants, then $G=g^{f_1,f_2}+f_3\mu\otimes\mu$ is parallel for $\tilde{\tilde{D}}^*=\tilde{D}^*+K$ where
\begin{equation}\label{Asolucao}
 K_XY=\biggl(\frac{f_3}{f_1}\theta^tX-\frac{r^2f_3^2}{(r^2f_3+f_1)f_1}\Omega(X)\theta^t\xi
\biggr)\mu(Y)
\end{equation}
and where $r^2=\|\xi\|^2$ and $\Omega(X)=\frac{1}{r^2}\langle\xi,X^v\rangle$  so that $\mu(\theta^tX)=\Omega(X)r^2$.\\
(iii) The connection $\na^{*,f_2,'}_XY=\na^*_XY+X(\varphi_2)Y$ is also metric on $(V,f_2\pi^*g)$.
\end{prop}
\begin{proof}
(i) The first assertion was proved earlier.\\
(ii) Since
\begin{eqnarray*}
 \na^{*,f_1}_X\mu\:Y &=&X(\mu Y)-\mu(\na^{*,f_1}_XY)\\
&=& \langle \na^*_X(\theta Y),\xi\rangle+
\langle\theta Y,\na^*_X\xi\rangle-\langle\theta\na^{*,f_1}_XY,\xi\rangle\\
&=& \langle \theta Y,X\rangle-X(\varphi_1)\mu(Y)-Y(\varphi_1)\mu(X)+\langle X^h,Y\rangle\partial\varphi_1
\end{eqnarray*}
we find $\tilde{D}^*_X(f_3\mu\otimes\mu)=X(f_3)\mu\otimes\mu+
f_3\na^{*,f_1}_X\mu\odot\mu =X(f_3)\mu\otimes\mu+f_3\bigl(X^\flat\circ\theta-X(\varphi_1)\mu-\mu(X)\dx\varphi_1+\partial\varphi_1.(X^h)^\flat\bigr)\odot\mu$.
So if $f_3$ and $f_1$ are constant, this derivative becomes $f_3X^\flat\circ\theta\odot\mu$. Notice $X^\flat\theta=(\theta^tX)^\flat$. Now writing $\tilde{\tilde{D}}^*=\tilde{D}^*+K$ we find
\begin{eqnarray*}
-\tilde{\tilde{D}}^*_XG(Y,Z)&=& -(\tilde{D}^*_XG+K_X\cdot G)(Y,Z)\\
& = &f_1\langle K_X^hY,Z^h\rangle+f_2\langle K_X^vY,Z^v\rangle-f_3\langle\theta^tX,Y\rangle\mu(Z)\\
& & -f_3\langle\theta^tX,Z\rangle\mu(Y)+f_3\mu(K_XY)\mu(Z)+f_3\mu(Y)\mu(K_XZ)+\\
& &+f_1\langle Y^h,K^h_XZ\rangle+f_2\langle Y^v,K_X^vZ\rangle.
\end{eqnarray*}
which has the solution given in (\ref{Asolucao}); notice in particular $K^v=0$.\\
(iii) We have for any vector $X$
\begin{eqnarray*}
\na^{*,f_2,'}_Xf_2\pi^*g&=& X(f_2)\pi^*g+f_2\na^{*,f_2,'}_X\pi^*g\\
&=& X(f_2)\pi^*g+f_2\na^{*}_X\pi^*g-2f_2X(\varphi_2)\pi^*g=0
\end{eqnarray*}
using $X(f_2)=2f_2X(\varphi_2)$.
\end{proof}
As the reader shall see, the last connection of the three is more relevant than the other given on $V$. We remark
\begin{equation}
 \na^{*,f_2}_XY-\na^{*,f_2,'}_XY=\langle\theta\grad\varphi_2,Y\rangle\theta X-\langle\theta X,Y\rangle\theta\grad\varphi_2.
\end{equation}
We have now a metric connection for each of the two cases mentioned above. With some extra work it is possible to find the Levi-Civita connection. However, for the moment, it seems rather cumbersome to study the analogous metric of the Cheeger-Gromoll metric, referring here to the extra weight $f_3\mu\otimes\mu$ on $H$ instead of $V$. Although this new metric is non-trivial on the tangent sphere bundles.

\subsection{The weighted metric $g^{f_1,f_2}$}

We shall proceed with the metric $G=g^{f_1,f_2}$. Recall this metric is supported by the decomposition $H\oplus V$ and $H$ depends on $\na$. Moreover the projections ${\cdot}^h$ and $\cdot^v$ act accordingly. We now give a generalization of \cite[Theorem 3.1]{Alb1}.

First, we recall the metric connection on $V$:
\begin{equation}
 \na^{*,f_2,'}_XY=\na^*_XY+X(\varphi_2)Y
\end{equation}
and thus we define $D^*$ as
\begin{equation}
 D^*=\na^{*,f_1}\oplus\na^{*,f_2,'}
\end{equation}
and define a tensor $B\in\Omega^1(\End{TTM})$ by
\begin{equation}
 B(X,Y)=Y(\varphi_2)X^v-\frac{f_2}{f_1}\langle X^v,Y^v\rangle\grad\varphi_2,
\end{equation}
where $\grad\varphi_2$ is the horizontal lift of the gradient. Also we let $R^*=\pi^*R^{\na}=R^{\pi^*\na}$ denote the curvature tensor of $\na^*$ and let $\calR=R^*\xi$. Notice $\calR(X,Y)=\calR(X^h,Y^h)$ and that we have $\calR\in\Omega^2(V)$. Then we let $A$ and $\tau$ be $H$-valued tensors defined respectively by
\begin{equation}\label{AlcTM}
f_1\langle A_XY,Z\rangle=\frac{f_2}{2}\bigl(\langle R^*_{X^h,Z^h}\xi, Y^v\rangle + \langle R^*_{Y^h,Z^h}\xi, X^v\rangle\bigr)
\end{equation}
and
\begin{equation}\label{taulcTM}
\tau(X,Y,Z)=\langle \tau_XY,Z^h\rangle= \dfrac{1}{2}\bigl(T(Y,X,Z)+T(X,Z,Y)+T(Y,Z,X)\bigr),
\end{equation}
with $T(X,Y,Z)=\langle \pi^*T^\na(X,Y),Z\rangle$ for any vector fields $X,Y,Z$ over $TM$, cf. (\ref{tauzero}).

Notice $\calR(X,Y)$ and $\tau(X,Y,Z)$ vanish if one of the directions $X,Y$ or $Z$ is vertical, whereas with $A(X,Y)$ the same happens if both $X,Y$ are vertical or both are horizontal. Hence, $A$ could be defined simply by $f_1\langle A_XY,Z\rangle=\frac{f_2}{2}(\langle\calR_{X,Z},Y\rangle
+\langle\calR_{Y,Z}, X\rangle)$.

\begin{prop}\label{torsaodenablaoplusnabla}
 The torsion of $\na^*\oplus\na^*$ is $\pi^*T^\na+\calR$.
\end{prop}
The proof of this essential equation is within the lines of the following result.
\begin{teo}\label{lcTM_0}
The Levi-Civita connection $\na^G$ of $TM$ with metric $G=g^{f_1,f_2}$ is given by
\begin{equation}\label{lcTM}
\na^G_XY=D^*_XY-\dfrac{1}{2}\calR(X,Y)+A(X,Y)+B(X,Y)+\tau(X,Y)
\end{equation}
$\forall X,Y$ vector fields over $TM$.
\end{teo}
\begin{proof}
Let us assume the identity and first see the horizontal part of the torsion:
\begin{eqnarray*}
\dx \pi(T^{\na^G}(X,Y))& =& D^*_XY^h+A_XY+B^h_XY+\tau_XY\\
& &\qquad -D^*_YX^h-A_YX-B^h_XY-\tau_YX-\dx\pi[X,Y] \\ 
&=& \pi^*T^\na(X,Y)+\tau_XY-\tau_YX,
\end{eqnarray*}
since this is how the torsion tensor of $\na$ lifts to $\pi^*TM$ and since $A$ and $B^h$ are symmetric tensors. Also recall $C_1$ is symmetric, so the torsion $T^\na=T^{\na^{f_1}}$. Now we check the vertical part:
\begin{eqnarray*}
(T^{\na^G}(X,Y))^v& = & D^*_XY^v-\dfrac{1}{2}R^*_{X,Y}\xi+B^v_XY-D^*_YX^v+ \dfrac{1}{2}R^*_{Y,X}\xi-B^v_YX-[X,Y]^v\\
& = &\na^*_X\na^*_Y\xi+X(\varphi_2)Y^v-R^*_{X,Y}\xi+Y(\varphi_2)X^v\\
& & -\na^*_Y\na^*_X\xi-Y(\varphi_2)X^v-X(\varphi_2)Y^v-\na^*_{[X,Y]}\xi\ =\ 0.
\end{eqnarray*}
$\na^G$ is a metric connection if and only if the difference with $D^*$ is skew-adjoint. Then, on one hand,
\begin{eqnarray*}
\lefteqn{G((\na^G-D^*)_XY,Z)\ =}\\
& =& -\dfrac{f_2}{2}\langle \calR_{X,Y},Z^v\rangle+f_1\langle A_XY
+\tau_XY,Z^h\rangle+ G(B_XY,Z)\\
& =& -\dfrac{f_2}{2}\langle R^*_{X,Y}\xi,Z^v\rangle +\dfrac{f_2}{2}\langle R^*_{X^h,Z^h}\xi, Y^v\rangle + \dfrac{f_2}{2}\langle R^*_{Y^h,Z^h}\xi, X^v\rangle+\\
& & +f_1\tau(X,Y,Z)+f_2Y(\varphi_2)\langle X^v,Z^v\rangle-f_2\langle X^v,Y^v\rangle\langle\grad\varphi_2,Z^h\rangle
\end{eqnarray*}
and, on the other,
\begin{eqnarray*}
\lefteqn{G((\na^G-D^*)_XZ,Y)\ =}\\ 
& =& -\dfrac{f_2}{2}\langle \calR_{X,Z},Y^v\rangle+f_1\langle A_XZ
+\tau_XZ,Y^h\rangle+ G(B_XZ,Y)\\
& =& -\dfrac{f_2}{2}\langle R^*_{X,Z}\xi,Y^v\rangle +\dfrac{f_2}{2}\langle R^*_{X^h,Y^h}\xi, Z^v\rangle + \dfrac{f_2}{2}\langle R^*_{Z^h,Y^h}\xi, X^v\rangle+\\
& &+f_1\tau(X,Z,Y)+f_2Z(\varphi_2)\langle X^v,Y^v\rangle-f_2\langle X^v,Z^v\rangle\langle\grad\varphi_2,Y^h\rangle
\end{eqnarray*}
hence the condition is expressed simply by $\tau(X,Y,Z)=-\tau(X,Z,Y)$. This, together with $\pi^*T^\na(X,Y)+\tau_XY-\tau_YX=0$, determines $\tau$ uniquely as the form given by (\ref{taulcTM}).
\end{proof}
It is clear that $H$ corresponds to an integrable distribution if and only if the connection $\na$ is flat. Indeed, the vertical part of $[X,Y]=\na^G_XY-\na^G_YX$, for any pair of horizontal vector fields, is $\calR(X,Y)=R^*_{X,Y}\xi$.

A first geometric consequence is at hand.
\begin{coro}
 The fibres $T_xM,\ \,x\in M$, are totally geodesic submanifolds of $TM$ if and only if $f_2$ is a constant.

The zero section of $TM$, i.e. the embedding $M\subset TM$, is totally geodesic if and only if $R^\na=0$.
\end{coro}
\begin{proof}
In view of the observations prior to the theorem, if $X,Y$ are two vertical vector fields, then $\na^G_XY=\na^*_XY-\frac{f_2}{f_1}\langle X^v,Y^v\rangle\grad\varphi_2$. Having this again in $\Gamma(V)$ is equivalent to the condition of each fibre being a totally geodesic submanifold. We immediately see that $\na^G_XY$ is a vertical vector field if and only if $\grad\varphi_2=0$. The question for the trivial horizontal section is solved analogously.
\end{proof}

It is important to understand when the tensor $\tau$ vanishes. We have the following result:
\begin{equation}\label{tauzero}
\tau=0\qquad \mbox{if and only if}\qquad T^\na=0.
\end{equation}
Indeed, if $\tau=0$, then $T(Y,X,Z)=T(Z,X,Y)+T(Z,Y,X)$; by the symmetries in $X,Y$ this tensor vanishes.
\begin{Rema}
By a result of \'E. Cartan, cf. \cite{Agri}, it is known that the space of torsion tensors $\Lambda^2TM\otimes TM$ of a metric connection decomposes into irreducible subspaces like
\begin{equation}\label{decompoftorsion}
{\cal A}\oplus\Lambda^3TM\oplus TM,
\end{equation}
where $\Lambda^3$ is the one for which $\langle T^\na(X,Y),Z\rangle$ is completely skew-symmetric and where $TM$ is the subspace of vectorial type torsions, i.e. for which there exists a vector field $V$ such that $T^\na(X,Y)=\langle V,X\rangle Y-\langle V,Y\rangle X$. The invariant subspace $\cal A$ is the orthogonal to those two. We also remark that, in dimension 4, under the special orthogonal group the space $\cal A$ is further decomposable in two 8 dimensional subspaces. Since $\Lambda^3TM^4$ is 4 dimensional, there is a second type of both vectorial and skew-symmetric torsion. This result has had consequences in \cite{Alb2}. 
\end{Rema}

\subsection{Almost Hermitian structure}
\label{comandsympstru}

We continue the study of $TM$ with the metric $G=g^{f_1,f_2}$ where $f_1=\expo^{2\varphi_1}$ and $f_2=\expo^{2\varphi_2}$. We let $\na$ denote a metric connection on $M$ with torsion $T^\na$. Some authors have studied an almost complex structure over $TM$ compatible with the Sasaki metric $g^S$ which was first discovered by Sasaki, cf. \cite{Dom,Sasa}. It may be written as the bundle endomorphism $I^S=\theta^t-\theta$, see (\ref{aplicacaotheta}). We call $(g^S,I^S)$ the Sasaki structure of $TM$, with torsion.

Some properties of the Sasaki metric related with its Hermitian structures $I,J$ or $K=IJ$ and quaternionic-Hermitian structure $(I,J,K)$, given by the natural almost complex structure $I=I^S$ and by an almost complex structure $J$ on $M$ pulled-back as $J\oplus J$, were studied in \cite{Alb1}. There we also admitted a metric connection with torsion for the study of $I^S$. We had in view the quaternionic-K\"ahler structure on $TM$, and may be generalized into the present setting too. In the next Theorem we need a formula from \cite{Alb1}.

Let 
\begin{equation}
 \psi=\varphi_2-\varphi_1,\qquad\quad\ \ \overline{\psi}=\varphi_2+\varphi_1.
\end{equation}
We then define an endomorphism $I^G$ by $I^GX=\expo^\psi\theta^tX-\expo^{-\psi}\theta X$ for all $X\in TTM$. Also we consider the associated \textit{symplectic} structure $\omega^G$, defined by
\begin{equation}
 \omega^G(X,Y)=G(I^GX,Y).
\end{equation}
\begin{prop} 
$I^G$ is an almost complex structure compatible with the metric $G$. The associated symplectic 2-form satisfies
\begin{equation}\label{ident2formas}
 \omega^G=\expo^{\overline{\psi}}\omega^S.
\end{equation}
\end{prop}
\begin{proof}
Indeed $(I^G)^2=(\expo^\psi\theta^t-\expo^{-\psi}\theta)
(\expo^\psi\theta^t-\expo^{-\psi}\theta)=-\theta^t\theta-\theta\theta^t=-1$.
And 
\begin{eqnarray*}
 G(I^GX,I^GY) &=& \expo^{2\psi}f_1\langle\theta^tX,\theta^tY\rangle+ 
\expo^{-2\psi}f_2\langle\theta X,\theta Y\rangle\\
&=&   \expo^{2\varphi_2}\langle X^v,Y^v\rangle+\expo^{2\varphi_1}\langle X^h,Y^h\rangle
\end{eqnarray*}
and this is clearly $G(X,Y)$. Since $f_1\expo^\psi=f_2\expo^{-\psi}=\expo^{\overline{\psi}}$, we easily get the conformality of $\omega^G$ with the Sasaki structure.
\end{proof}
\begin{teo}
(i) The almost complex structure $I^G$ is integrable if and only if $\na$ is flat and 
\begin{equation}
 T^\na=\dx\psi\wedge 1
\end{equation}
or equivalently $T^\na(X,Y)=X(\psi)Y-Y(\psi)X,\ \forall X,Y\in TM$. It is thus a vectorial torsion type metric connection.

In particular, if $\na$ is torsion free, then $I^G$ is integrable if and only if $M$ is Riemannian flat and $f_2/f_1=$constant.\\
(ii) $(TM,\omega^G)$ is a symplectic manifold if and only if
\begin{equation}
 T^\na=\dx\overline{\psi}\wedge1.
\end{equation}
In particular, with $\na$ the Levi-Civita connection, $\dx\omega^G=0$ if and only if $f_2f_1=$constant.
\end{teo}
\begin{proof}
(i) Let $i=\sqrt{-1}$ and let us denote $I^G=I$. As it is well known, if for all $v,w$ in the $+i$-eigenbundle of $I$ we have $\na^G_wv$ in the $+i$-eigenbundle, then $[w,v]=\na^G_wv-\na^G_vw$ will be in the very same space and the $I$ structure will be integrable by the well known Newlander-Niremberg's Theorem. Reciprocally, the integrability of $I$ implies the first condition on the Riemannian connection (the proof in this general setting is simple, cf. \cite{Alb1} or the original reference by S. Salamon \cite{Salamon}).

Recall the Levi-Civita connection for $G$ is $\na^G=D^*-\frac{1}{2}\calR+A+B+\tau$, with these tensors given in Theorem \ref{lcTM_0}.
Let $X,Y$ be any real vector fields on $TM$. Let $w=X-iIX$ and $v=Y-iIY$. Then
\begin{equation*}
 \na^G_wv\ =\ \na^G_{X-iIX}{(Y-iIY)}\ =\ \na^G_XY-\na^G_{IX}{IY}-i(\na^G_{IX}Y+\na^G_XIY).
\end{equation*}
Now suppose $X,Y$ are horizontal vector fields. Then 
\[ A(X,Y)=0,\ \ \ B(X,Y)=0,\ \ \ B(X,\theta Y)=0 \]
\[  B(\theta X,Y)=Y(\varphi_2)\theta X,\ \ \ B(\theta X,\theta Y)=-\expo^{2\psi}\langle X,Y\rangle\grad\varphi_2  \]
and hence
\begin{eqnarray*}
\na^G_wv &=& \na^G_XY-\expo^{-\psi}\na^G_{\theta X}\expo^{-\psi}\theta Y+
i(\na^G_X\expo^{-\psi}\theta Y+\expo^{-\psi}\na^G_{\theta X}Y)\\
&=& \na^{f_1}_XY-\frac{1}{2}\calR(X,Y)+\tau(X,Y)-\expo^{-2\psi}(\na_{\theta X}\theta Y+B(\theta X,\theta Y))+\\
& & \ \ \ \ +i\expo^{-\psi}\bigl(-X(\psi)\theta Y+\na_X\theta Y+X(\varphi_2)\theta Y
+A(X,\theta Y)+\\
& &\qquad\ \ \ \ \  \na^{f_1}_{\theta X}Y+A(\theta X,Y)+B(\theta X,Y)\bigr) \\
&=& \na_XY+\pi^*C_1(X,Y)-\frac{1}{2}\calR(X,Y)+\tau(X,Y)-\expo^{-2\psi}\theta\na_{\theta X}Y +\langle X,Y\rangle\grad\varphi_2+\\
& & \ \ \ \ +i\expo^{-\psi}\bigl(X(\varphi_1)\theta Y+\theta\na_XY+A(X,\theta Y)+\na_{\theta X}Y+A(\theta X,Y)+Y(\varphi_2)\theta X\bigr) 
\end{eqnarray*}
because $(\theta X)(\psi)=0$, because $\theta$ is $\na$-parallel and $\pi^*C_1$ only depends on horizontals. Now
\begin{eqnarray*}
\mathrm{Re}\,I\na^G_wv
&=& -\expo^{-\psi}(\theta\na_XY+X(\varphi_1)\theta Y+Y(\varphi_1)\theta X-\langle X,Y\rangle\theta\grad\varphi_1)   \\
& &-\frac{1}{2}\expo^\psi\theta^t\calR(X,Y)-\expo^{-\psi}\theta\tau(X,Y)-\expo^{-\psi}\na_{\theta X}Y-\expo^{-\psi}\langle X,Y\rangle\theta\grad\varphi_2
\end{eqnarray*}
and
\begin{equation*}
\mathrm{Re}\,i\na^G_wv = -\expo^{-\psi}\bigl(X(\varphi_1)\theta Y+\theta\na_XY+A(X,\theta Y)+\na_{\theta X}Y+A(\theta X,Y)+Y(\varphi_2)\theta X\bigr) 
\end{equation*}
Finally putting in equation, $\mathrm{Re}\,(I-i1)\na^G_wv=0$, it is easy to see the terms appearing with $\na$ cancel. So we are left with
\begin{eqnarray*}
 & & X(\varphi_1)\theta Y+Y(\varphi_1)\theta X+\langle X,Y\rangle\theta\grad(\varphi_2-\varphi_1)+\frac{1}{2}\expo^{2\psi}\theta^t\calR(X,Y)+\\
 & &  \hspace*{3cm}+\theta\tau(X,Y)=
X(\varphi_1)Y+A(X,\theta Y)+A(\theta X,Y)+Y(\varphi_2)\theta X.
\end{eqnarray*}
Looking at horizontal and vertical parts,
\[\left\{\begin{array}{l}
         \frac{1}{2}\expo^{2\psi}\theta^t\calR(X,Y)=A(\theta X,Y)+A(X,\theta Y)\\
Y(\varphi_1)X+\langle X,Y\rangle\grad(\varphi_2-\varphi_1)+\tau(X,Y) =Y(\varphi_2)X     \end{array} \right. .  \]
Let us see the first equation: 
$\theta(A(\theta X,Y)+A(X,\theta Y))=\frac{\expo^{2\psi}}{2}\calR(X,Y)$. Following from the very definition of $A$ we have
\[ \langle\theta A((\theta X,Y)+A(X,\theta Y)),\theta Z\rangle=\frac{\expo^{2(\varphi_2-\varphi_1)}}{2}(\langle\calR(Y,Z),\theta X\rangle+\langle\calR(X,Z),\theta Y\rangle) . \]
The two equations combine and on the base $M$ it is easy to see they read
\[ R(Y,Z)X+R(X,Z)Y=R(X,Y)Z . \]
The symmetries of $R$ imply $R=0$, notice independently of the Bianchi identity. The second equation reads
$\tau(X,Y)=Y(\psi)-\langle X,Y\rangle\grad\psi$. Since 
\[ -T(X,Y)=\tau(X,Y)-\tau(Y,X)=Y(\psi)X-X(\psi)Y  \]
we find $T^\na=\dx\psi\wedge 1$. In particular for the Sasaki metric we get the already known result.

The imaginary part of $I\na^G_wv=i\na_wv$ gives an equivalent condition, since we may use the above and change $iw$ for $w$. Notice we have used $X,Y\in H$. It is enough, since the projection
$X\rightsquigarrow X-iIX=X+i\expo^{-\psi}\theta X$ becomes a $\C$-isomorphism between $H\otimes\C$ and the $+i$-eigenbundle of $I$. This proves the sufficiency of the condition in order to have integrability.\\
(ii) From (\ref{ident2formas}) we get
\[  \dx\omega^G=\expo^{\overline{\psi}}(\dx\overline{\psi}\wedge\omega^S+\dx\omega^S) .  \]
Now we need to choose a basis of $g$-orthonormal vectors $e_i$ together with their mirror images $e_{i+m}=\theta e_i$,\ $i=1,\ldots,m$. From \cite{Alb1} we find the formula:
\begin{eqnarray*}
  \dx\omega^S &=& \sum_{i<j<k}^m\bigl(\langle\calR(e_k,e_i),\theta e_j\rangle+\langle\calR(e_j,e_k),\theta e_i\rangle+\langle\calR(e_i,e_j),\theta e_k\rangle\bigr)e^{ijk}+\\
& & \ \ \ +\sum_{i<j}^m\sum_{k=1}^m(\tau_{ijk}-\tau_{jik})e^{ij,k+m}
\end{eqnarray*}
where $e^{ijk}=e^i\wedge e^j\wedge e^k$. Since $\tau_{ijk}-\tau_{jik}=-T_{ijk}$ and the curvature components do not involve vertical indices, the equation $\dx\omega^G=0$ is satisfied under the conditions
\[ \left\{\begin{array}{l}
           \cyclic_{ijk}R_{ijk}=0\\ 
\dx\overline{\psi}(e_i)e^{ij,j+m}-T_{ijk}e^{ij,k+m}=0           \end{array}\right.  .     \]
That is, the Bianchi identity and $T=\dx\overline{\psi}\wedge1$. Finally we recall a result stated in \cite{Agri}. A metric connection with vectorial torsion $V$ satisfies
\[  \cyclic_{X,Y,Z}R(X,Y)Z=\cyclic_{X,Y,Z}\dx V(X,Y)Z . \]
In our case, $V$ is a gradient, hence $\dx V=\dx\dx\overline{\psi}=0$ and thence Bianchi identity is immediately satisfied.
\end{proof}
The Theorem above suggests some observation. In the strict Sasaki metric case we had $T^\na=0$ as necessary condition of both integrability of $I^S$ and $\dx\omega^S=0$. In the general case, things are distinguished, as they should, by $\psi$ and $\overline{\psi}$.

Clearly we may draw the following conclusion.
\begin{coro}\label{coroKahlerTMflat}
$(TM,G,I^G,\omega^G)$ is K\"ahler if and only if $(M,\na)$ is a Riemannian flat manifold ($T^\na=0,\ R^\na=0$) and $f_1,f_2$ are constants. In this case, $TM$ is flat.
\end{coro}
The last assertion follows easily.

\subsection{A natural contact structure}

Recall $T^*M$ has a natural symplectic structure: $\dx\lambda$ where $\lambda$ is the Liouville 1-form, i.e. the unique 1-form $\lambda$ on $T^*M$ such that on a point $\alpha$ we have $\lambda_\alpha=\alpha\circ\pi_*$. Equivalently, such that $\alpha^*\lambda=\alpha$ for any section $\alpha\in\Omega^1_M$. Once we introduce the Riemannian structure, the tangent and co-tangent (sphere) bundles become isomorphic. We easily deduce that $\mu$ defined in (\ref{etaemu}) corresponds by that isomorphism to the Liouville form --- so it does not depend on the connection. 

By Proposition \ref{torsaodenablaoplusnabla} we know the torsion of $\na^*\oplus\na^*$ for any metric connection on $M$. It is then easy to deduce as in \cite{Alb2}, writing $T=\pi^*T^\na$: 
\begin{equation}\label{dmuigualaomegamaismuT}
 \dx\mu=\omega^S+\mu\circ T.
\end{equation}
The same is to say $\omega^S$ corresponds with the pull-back of the Liouville symplectic form if and only if $T^\na=0$. Notice $T(X,Y)$ vanishes if one direction $X$ or $Y$ is vertical.

Regarding the contact structure on $S_rM\subset TM$, as in classical Y. Tashiro \cite{Tash}, the restriction of the Liouville 1-form defines indeed a contact structure --- always, no matter the metric, the radius function or the metric connection. We follow e.g. \cite{Geig} for the definition.
\begin{teo}
 For any $r\in\cinf{M}(\R^+)$, the 1-form $\mu$ defines a contact structure on $S_rM$.
\end{teo}
\begin{proof}
Let $n=\dim M-1$ and let $e_0,\ldots,e_n$ be a local orthonormal basis of $TM$ with $e_0=u\in TM$ a generic point. We lift the frame and extend with $\theta e_0,\ldots, \theta e_n$ over $TM$. We may assume locally $\mu=e^0$. We denote $e_ {j+n}=\theta e_j$. Then $\omega^S=-\sum_{j=1}^ne^{j,j+n}$ (cf. \cite{Alb2} for these formulas). We may clearly write $\mu\circ T_u=\|u\|^2\sum_{0\leq i<j\leq n}T_{ij0}e^{ij}$. Let $\iota:S_rM\rr TM$ denote the inclusion map. Then
\[ \iota^*\mu\wedge(\dx\iota^*\mu)^n=\iota^*(\mu\wedge(\dx\mu)^n)  \]
so we may omit the $\iota$ in the following. With a moments thought, we see
\begin{eqnarray*}
 \mu\wedge(\dx\mu)^n=\mu\wedge\bigl(-\sum_{1\leq j\leq n}e^{j,j+n}
+r^2\sum_{0\leq i<j\leq n}T_{ij0}e^{ij}\bigr)^n=(-1)^{n-1}n!e^{012\cdots (2n)}.
\end{eqnarray*}
To see that this is $\neq 0$ on $S_rM$ we take a 1-form on $TM$ which has kernel $TS_rM$: \,$\Gamma=\xi^\flat-r\dx r$. Indeed, differentiating the hypersurface equation $\langle\xi,\xi\rangle-r^2=0$ with the aid of $\na^*$, we get the 1-form $\Gamma$. Finally,
\[ \Gamma\wedge\mu\wedge(\dx\mu)^n=(-1)^{n-1}n!\xi^\flat\wedge e^{012\cdots (2n)}\neq0 \]
since $\dx r$ is a horizontal 1-form. This implies $\mu\wedge(\dx\mu)^n\neq0$ over $S_rM$.
\end{proof}
For $r$ constant, a metric associated to $\mu$ is recovered as the Tashiro metric contact structure on $S_1M$ if and only if $T^\na=0$, due to (\ref{dmuigualaomegamaismuT}). Such contact structure is given by
\[  \tilde{g}=\frac{1}{4}g^S,\quad\eta=\frac{1}{2r}\mu,\quad\varphi=\theta-\theta^t-\frac{1}{r^2}\xi\otimes\mu,\quad\zeta=\frac{2}{r}\theta^t\xi   \]
in order to satisfy standard identities. $\zeta$ is the characteristic vector field and $\varphi$ is the associated (1,1)-tensor such that $\varphi^2=-1+\zeta\otimes\eta$ and $\varphi(\zeta)=0$. Notice
$\eta=\zeta\lrcorner\tilde{g}$, $\tilde{g}(\varphi,\varphi)=\tilde{g}-\eta\otimes\eta$ and $\dx\eta=2\tilde{g}(\ ,\varphi\ )$, as expected.

\end{document}